\documentclass[10pt,a4paper]{article}
\usepackage[utf8]{inputenc}
\usepackage{amsmath}
\usepackage{amsfonts}
\usepackage{amssymb}
\usepackage{amsthm}
\usepackage{color}
\usepackage{enumerate}
\usepackage{marvosym}
\usepackage{comment}
\usepackage{tikz-cd}
 \newtheorem{thm}{Theorem}[section]

 \newtheorem*{thm*}{Theorem}
 \newtheorem{cor}[thm]{Corollary}
 \newtheorem{lem}[thm]{Lemma}
 
 \theoremstyle{definition}
 \newtheorem{defn}[thm]{Definition}
 \theoremstyle{remark}
 \newtheorem{rem}[thm]{Remark}
 
 \newtheorem{ex}{Example}
 
 \numberwithin{equation}{section}
 
\newcommand{\vertiii}[1]{{\left\vert\kern-0.25ex\left\vert\kern-0.25ex\left\vert #1
  \right\vert\kern-0.25ex\right\vert\kern-0.25ex\right\vert}}

\newcommand{\tr}{\operatorname{Tr}}

\newcommand{\Span}{\operatorname{span}}

\newcommand{\linspan}{\operatorname{span}}
\newcommand{\ind}{\operatorname{ind}}

\title{Three short tales on the parity operator}
\author{Robert Fulsche}
\begin{document}
\maketitle
\begin{abstract}
    In this paper, we discuss three short topics related to the parity operator and his role in quantum harmonic analysis. We derive results for the Fredholm index of even and odd operators, discuss operators on which the modulation action acts continuous in operator norm and show that the parity operator plays a natural role in the operator-to-operator Fourier transform of quantum harmonic analysis.
\end{abstract}

\section{Introduction}

The parity operator is seemingly a rather straightforward operator: On $L^2(\mathbb R^n)$ it is defined as $Uf(t) = f(-t)$. This seemingly harmless operator is nevertheless of fundamental importance. It plays, for example, a key role in the mathematical formulation of quantum mechanics, as well as pseudodifferential operator theory and time-frequency analysis, and has been at the heart of several detailed discussions. We only mention the papers 
\cite{Canivell_Seglar1978, Epstein1965, Grossmann1976} as examples.

Equally well, the parity operator plays a crucial role in the framework of \emph{quantum harmonic analysis} \cite{Werner1984}. But its importance goes far beyond merely setting up the framework of quantum harmonic analysis (even though its importance there cannot be overestimated). Nevertheless, there are a few important aspects of the parity operator which have not been discussed in the literature so far. This note is dedicated to the discussion of three of such aspects. We very briefly give descriptions of these three topics:

In the beginning, we discuss the Fredholm theory of \emph{even and odd} operators. By this, we mean bounded linear operators which commute, respectively anticommute with the parity operator. In general, there is not much that one can say about such operators when it comes to their Fredholm spectrum. Nevertheless, when we assume that the operators are in $\mathcal C_1(\mathcal H)$, an algebra naturally occuring in quantum harmonic analysis, it turns out that even operators always have even Fredholm index and odd operators always have odd Fredholm index (provided they are Fredholm in the first place). This is the main result of the first topic.

In the second topic, we discuss operators on which the \emph{modulation action} of quantum harmonic analysis acts continuous in operator norm. As it turns out, composing with the parity operator reduces the study of such operators to the study of operators in the algebra $\mathcal C_1(\mathcal H)$ that we already mentioned in the previous paragraph. A little more surprisingly, it turns out that the class of operators which is contained in $\mathcal C_1(\mathcal H)$ and on which the modulation action acts continuously in operator norm is exactly the class of compact operators. 

In the last section, we will discuss the notion of operator-to-operator Fourier transform, which shows up naturally in the context of quantum harmonic analysis, but seemingly has not been noted before. In fact, we will show that this variant of the Fourier transform can be expressed (up to a multiplicative constant) by composing any operator simply with the parity operator. This observations can be utilized to give a very straightforward proof of the (well-known) fact that a Weyl pseudodifferential operator is contained in some operator ideal if and only if the Weyl pseudodifferential operator with the symplectic Fourier transform of the initial symbol being the new symbol is contained in the same operator ideal.

In addition to those three discussions, we will start the topic with a section on some preliminaries, setting up definitions and some basic results.

\section{Preliminaries}
The results discussed in this paper are closely connected with the setting of \emph{quantum harmonic analysis} of the group $\mathbb R^{2n}$, endowed with a symplectic form. This quantum harmonic analysis can be formulated in an abstract way in terms of certain projective unitary representations of $\mathbb R^{2n}$, but we will most of the time work in rather concrete settings, i.e., we will work with concrete realizations of these representations.
\begin{ex}
    \begin{enumerate}
        \item Let $\mathcal H = L^2(\mathbb R^n)$. Then, we consider the symmetric time-frequency shifts
        \begin{align*}
            W_{(x, \xi)}f(t) = e^{-i\xi \cdot t + i \frac{x \cdot \xi}{2}}f(t-x), \quad x, \xi, t \in \mathbb R^n, ~f \in L^2(\mathbb R^n).
        \end{align*}
        These unitary operators, also referred to as \emph{Weyl operators}, satisfy the exponentiated CCR relation:
        \begin{align*}
            W_{(x, \xi)} W_{(y, \eta)} = e^{-i\frac{\sigma((x, \xi), (y, \eta)}{2}} W_{(x+y, \xi+\eta)}, \quad (x, \xi), (y, \eta) \in \mathbb R^{2n}.
        \end{align*}
        Here, the symplectic form $\sigma$ is given by;
        \begin{align*}
            \sigma((x, \xi), (y, \eta)) = y\cdot \xi - x\cdot \eta.
        \end{align*}
        Usually, we will write $z = (x, \xi), w = (y, \eta)$ for points in the phase space $\mathbb R^{2n}$ to simplify the notation.
        \item Let $\mu$ be the Gaussian measure $d\mu(z) = \frac{1}{(2\pi)^n}e^{-\frac{|z|^2}{2}}~dz$ on $\mathbb C^n$. Then, the subspace $F^2(\mathbb C^n)$ of $L^2(\mathbb C^n, \mu)$ consisting of entire functions is a reproducing kernel Hilbert space. The reproducing kernel function $K_z(w)$ has the form:
        \begin{align*}
            K_z(w) = e^{\frac{w \cdot \overline{z}}{2}}, \quad z, w \in \mathbb C^n.
        \end{align*}
        The normalized reproducing kernels are then given by:
        \begin{align*}
            k_z(w) := \frac{K_z(w)}{\| K_z\|} = e^{\frac{w \cdot \overline{z}}{2}- \frac{|z|^2}{4}}, \quad z, w \in \mathbb C^n.
        \end{align*}
        When working in this setting, we identify the phase space $\mathbb R^{2n}$ with $\mathbb C^n$ in the natural way. Then, the Weyl operators $W_z$ are defined by:
        \begin{align*}
            W_z f(w) = k_z(w) f(w-z), \quad z, w \in \mathbb C^n, ~f \in F^2(\mathbb C^n).
        \end{align*}
        Also in this setting the Weyl operators satisfy:
        \begin{align*}
            W_z W_w = e^{-i\frac{\sigma(z,w)}{2}} W_{z+w}, \quad z, w \in \mathbb R^{2n} \cong \mathbb C^n.
        \end{align*}
        Note that the symplectic form can, in complex coordinates, also be expressed as $\sigma(z, w) = \operatorname{Im}(z \cdot \overline{w})$. A standard reference for analysis on the Fock space is \cite{Zhu} (which only deals with the case $n = 1$, but most results can be carried over to higher dimensions without troubles). 
    \end{enumerate}
\end{ex}
It is well-known that in both of these situations, the Weyl operators act irreducibly on $\mathcal H$. Hence, by the theorem of Stone and von-Neumann, both projective unitary representations are unitarily equivalent (for the same value of $n$). Furthermore, we want to note that in both situations one can also introduce a semiclassical parameter. Everything that will be discussed in this extension can be extended to every value of that semiclassical parameter without significant changes. For simplicity, we will nevertheless not carry that parameter through this paper.

The parity operator is now defined as a unitary operator $U$ satisfying the relation:
\begin{align*}
    W_z U = UW_{-z}, \quad z \in \mathbb R^{2n}.
\end{align*}
It is not hard to show that $U$ is unique (up to a multiplicative constant) and that $U$ can be chosen self-adjoint (see, for example, the discussion in \cite[Section 2]{Fulsche_Galke2025} on the parity operator). In concrete situations, it is not hard to check how the parity operator is realized:
\begin{ex}
    \begin{enumerate}
        \item In the situation with $\mathcal H = L^2(\mathbb R^n)$ described above, it is $Uf(t) = f(-t).$
        \item When $\mathcal H = F^2(\mathbb C^n)$, then the parity has essentially the same form: $Uf(z) = f(-z)$.
    \end{enumerate}
\end{ex}
An important space when it comes to the methods of quantum harmonic analysis is the $C^\ast$-algebra $\mathcal C_1(\mathcal H)$ defined by:
\begin{align*}
    \mathcal C_1(\mathcal H) := \{ A \in \mathcal L(\mathcal H): ~\| W_z A W_z^\ast - A\|_{op} \to 0, ~z \to 0\}.
\end{align*}
This algebra will appear prominently throughout the following sections. The expression $W_z A W_z^\ast$ is often termed the \emph{shift} of the operator $A$ and denoted 
\begin{align*}
    \alpha_z(A) := W_z A W_z^\ast
\end{align*} 
for $A \in \mathcal L(\mathcal H)$ and $z \in \mathbb R^{2n}$. Hence, $\mathcal C_1(\mathcal H)$ consists exactly of those $A \in \mathcal L(\mathcal H)$ such that $z \mapsto \alpha_z(A)$ is continuous in operator norm.

Further concepts from quantum harmonic analysis or operator theory will be introduced through the paper whenever it is necessary. We just want to mention that basic references to the theory that we use here are \cite{Werner1984, Luef_Skrettingland2018, Fulsche2020, Fulsche_Galke2025}.

\section{Even and odd operators}
In this section, we will work with $\mathcal H = F^(\mathbb C^n)$, the Fock space. By unitary equivalence, the main results then carry over to any other representation equivalent to the Fock space representation, in particular to the representation by symmetric time-frequency shifts on $L^2(\mathbb R^n)$.

\begin{defn}
    A bounded linear operator $A \in \mathcal L(\mathcal H)$ is called \emph{even} if $UAU = A$ and \emph{odd} if $UAU = -A$.
\end{defn}
In the following, we will not only work with even and odd operators, but also with even and odd functions. For this, we introduce the notation $\beta_-(f)$ for the function $\beta_-(f)(z) = f(-z)$. Hence, a function $f$ on $\mathbb C^n$ is even if $\beta_-(f) = f$ and odd if $\beta_-(f) = -f$.
\begin{ex}
    We denote by $P$ the orthogonal projection from $L^2(\mathbb C^n, \mu)$ to $F^2(\mathbb C^n)$. Then, for $f \in L^\infty(\mathbb C^n)$ the Toeplitz operator $T_f \in \mathcal L(F^2(\mathbb C^n))$ is defined by $T_f(g) = P(fg)$. Toeplitz operators satisfy the norm estimate $\| T_f\|\leq \| f\|_\infty$. 
    
   The Toeplitz operator $T_f$ is even if and only if $f$ is even and $T_f$ is odd if and only if $f$ is odd. Seeing that $f$ being even/odd is sufficient is straightforward. To see that this is also necessary for $T_f$ being even/odd, note that a bounded linear operator $A\in \mathcal L(F^2(\mathbb C^n))$ is even/odd if and only if its Berezin transform $\widetilde{A}(z) := \langle A k_z, k_z\rangle$ is even/odd (which is shown using the fact that $\widetilde{A} = 0$ if and only if $A = 0$). For a Toeplitz operator, $\widetilde{T_f}$ is given by $f \ast g$, where $g$ is the Gaussian $g(z) = \frac{1}{(2\pi)^n}e^{-\frac{|z|^2}{2}}$. Since $\beta_-(f \ast g) = \beta_-(f) \ast \beta_-(g)$ and $g$ is even ($\beta_-(g) = g$), as well as combining this with the fact that convolution with a Gaussian is injective, this concludes the statement.
\end{ex}
Every operator admits a canonical decomposition into an even and an odd operator:
\begin{lem}
    Let $A \in \mathcal L(\mathcal H)$.
    \begin{enumerate} 
    \item $A + UAU$ is even and $A - UAU$ is odd.
    \item $A$ can be written as $A = A_{even} + A_{odd}$, where $A_{even}$ is an even operator and $A_{odd}$ is an odd operator.
    \item The decomposition in (2) is unique. In particular, $A_{even} = \frac 12 (A + UAU)$ and $A_{odd} = \frac 12 (A - UAU)$.
    \end{enumerate}
\end{lem}
\begin{proof}
    All statements but the uniqueness are entirely obvious. Regarding the uniqueness, note that for any decomposition $A = A_{even} + A_{odd}$ into even and odd operators we also have $UAU = A_{even} - A_{odd}$. Adding both equations yields $A + UAU = 2A_{even}$ and subtracting both equations yields $A - UAU = 2A_{odd}$.
\end{proof}
Note that the parity operator $U$ has spectrum $\sigma(U) = \{ -1, 1\} = \sigma_{ess}\{ -1, 1\}$. Hence, the Hilbert space can be decomposed as $\mathcal H = \mathcal H_{even} \oplus \mathcal H_{odd}$, where $\mathcal H_{even}$ is the eigenspace of $U$ with respect to the eigenvalue $1$ and $\mathcal H_{odd}$ is the eigenspace with respect to the eigenvalue $-1$. More concretely, letting $e_\alpha(z) = \frac{z^\alpha}{\sqrt{2^n \alpha!}}$ denote the standard basis elements, we have $\mathcal H_{even} = \oplus_{|\alpha| \text{even}} \linspan \{e_\alpha\}$ and $\mathcal H_{odd} = \oplus_{|\alpha| \text{odd}} \linspan \{e_\alpha\}$. The elements of $\mathcal H_{even}$ and $\mathcal H_{odd}$ are, respectively, determined by the equations $Uf(z) = f(z)$ and $Uf(z) = -f(z)$, respectively. A straightforward consequence is the following:
\begin{lem}
    Let $A \in \mathcal L(\mathcal H)$. 
    \begin{enumerate}
        \item If $A$ is even, then $A\mathcal H_{even} \subset \mathcal H_{even}$ and $A\mathcal H_{odd} \subset \mathcal H_{odd}$.
        \item If $A$ is odd, then $A\mathcal H_{even} \subset \mathcal H_{odd}$ and $A\mathcal H_{odd} \subset \mathcal H_{even}$. 
    \end{enumerate}
\end{lem}
Any bounded linear operator $A \in \mathcal L(\mathcal H)$ can be written as an operator matrix with respect to the decomposition $\mathcal H = \mathcal H_{even} \oplus\mathcal H_{odd}$: 
\begin{align*}
    A \cong \begin{pmatrix}
        A_{11} & A_{12} \\ A_{21} & A_{22}
    \end{pmatrix},
\end{align*}
where $A_{11} \in \mathcal L(\mathcal H_{even})$, $A_{12} \in \mathcal L(\mathcal H_{odd}, \mathcal H_{even}), A_{21} \in \mathcal L(\mathcal H_{even}, \mathcal H_{odd})$ and $A_{22} \in\mathcal L(\mathcal H_{odd})$. 
\begin{lem}
    $A$ is even if and only if $A_{12} = 0$ and $A_{21} = 0$. $A$ is odd if and only if $A_{11} = 0$ and $A_{22} = 0$. 
\end{lem}

The remainder of this section will be dedicated to a property of the Fredholm index of even and odd operators. Clearly, for arbitrary even/odd operators, no significant statements can be made: For example, the Fredholm index of an even operator can in general be any integer. To see this, simply consider the even operator which is, in its matrix decomposition, given by
\begin{align*}
    \begin{pmatrix}
        I_{\mathcal H_{even}} & 0\\ 0 & A_k
    \end{pmatrix},
\end{align*}
where $A_k$ is any Fredholm operator on $\mathcal H_{odd}$ with Fredholm index $k \in \mathbb Z$. Then, the index of this operator is of course also $k$. 

Therefore, we now restrict to operators in the $C^\ast$-algebra $\mathcal C_1(\mathcal H)$. We recall that this $C^\ast$-algebra agrees with the $C^\ast$-algebra generated by all Toeplitz operators on the Fock space with bounded symbols \cite[Theorem 3.1]{Fulsche2020}.

For a $C^\ast$-subalgebra $\mathcal A$ of $\mathcal L(\mathcal H)$, we define its \emph{essential center} as:
\begin{align*}
    \operatorname{essCen}(\mathcal A) := \{ A \in \mathcal L(\mathcal H): ~[A, B] \in \mathcal K(\mathcal H) \text{ for every } B \in \mathcal A\}
\end{align*}
Since the $C^\ast$-algebra generated by all Toeplitz operators with bounded symbols agrees with $\mathcal C_1(\mathcal H)$, the following result is just a reformulation of \cite[Theorem D]{Berger_Coburn1987}:
\begin{thm}
    $\operatorname{essCen}(\mathcal C_1(\mathcal H)) = \{ T_f: ~f \in \operatorname{VO}_\partial (\mathbb C)\} + \mathcal K(\mathcal H)$
\end{thm}
In the following, we will make use of one particular Toeplitz operator. The choice of this operator is, to some degree, arbitrary, but it satisfies a collection of properties that we will need. 
\begin{lem} The Toeplitz operator $T_{z/|z|}$ satisfies the following properties:
\begin{enumerate}
    \item It is contained in $\operatorname{essCen}(\mathcal C_1(\mathcal H))$.
    \item It is odd.
    \item Let its matrix representation with respect to the decomposition $\mathcal H = \mathcal H_{even} \oplus \mathcal H_{odd}$ be given by \begin{align*}
    T_{z/|z|} = \begin{pmatrix}
        0 & B_1 \\ B_2 & 0
    \end{pmatrix}.
\end{align*}
Then, $B_2$ is bijective. $B_1$ is injective and Fredholm with index $-1$. In particular, $T_{z/|z|}$ is Fredholm with index $-1$.
\end{enumerate}
\end{lem}
\begin{proof} 
That $T_{z/|z|}$ is contained in $\operatorname{essCen}(\mathcal C_1(\mathcal H))$ can be verified easily, using the above characterization of the essential center. Since $z \mapsto z/|z|$ is an odd function, the operator is odd. Finally, by \cite[Lemma 18]{Al-Qabani_Virtanen}, $T_{z/|z|}$ is a weighted shift operator, i.e., acts diagonally on the standard orthonormal basis, sending $e_m$ to $\alpha_m e_{m+1}$ with $\alpha_m \neq 0$ and $a_m \to 1$ as $m \to \infty$. In particular, for the matrix entries we have that $B_2: \mathcal H_{even} \to \mathcal H_{odd}$ is bijective and $B_1: \mathcal H_{odd} \to \mathcal H_{even}$ is injective with $\operatorname{Ran}{B_1} = \overline{\Span}\{ e_2, e_4, e_6, \dots\}$. Hence, $B_1$ and $B_2$ are Fredholm with $\ind(B_1) = -1$ and $\ind(B_2) = 0$. 
\end{proof}
\begin{thm}
    Let $A \in \mathcal C_1(\mathcal H)$ be even. If $A$ is Fredholm, then its index is even.
\end{thm}
\begin{proof}
    Since $A$ is even, its matrix representation is of the form
    \begin{align*}
        A = \begin{pmatrix}
            A_1 & 0 \\ 0 & A_2
        \end{pmatrix}
    \end{align*}
    and $A$ is Fredholm if and only if both $A_1$ and $A_2$ are Fredholm. In this case, $\ind(A) = \ind(A_1) + \ind(A_2)$. We will show that $\ind(A_1) = \ind(A_2)$, which proves the result.

    Since $T_{z/|z|} \in \operatorname{essCen}(\mathcal C_1(\mathcal H))$, using also the fact that $T_f^\ast = T_{\overline{f}}$ for every $f \in L^\infty(\mathbb C^n)$, we observe that
    \begin{align*}
        \begin{pmatrix} B_2^\ast A_2 B_2 & 0 \\ 0 & B_1^\ast A_1 B_1 \end{pmatrix} = T_{z/|z|} A  T_{\overline{z}/|z|} = A + K.
    \end{align*}
    In particular, $A_1 - B_2^\ast A_2 B_2 \in \mathcal K(\mathcal H_{even})$ such that
    \begin{align*}
        \ind(A_1) = \ind(B_2^\ast A_2 B_2) = \ind(A_2),
    \end{align*}
    which finishes the proof.
\end{proof}
\begin{cor}
    Let $A \in \mathcal C_1(\mathcal H)$ be odd. If $A$ is Fredholm, then $\ind(A)$ is odd.
\end{cor}
\begin{proof}
    When $A$ is odd, then $AT_{z/|z|}$ is even. Hence, 
    \begin{align*}
        \ind(A) = \ind(AT_{z/|z|}) - \ind(T_{z/|z|}) = \ind(AT_{z/|z|}) + 1.
    \end{align*}
    Since $AT_{z/|z|}$ is even, the result follows from the previous theorem. 
\end{proof}
The approach described here works more generally, in the following setting. For $k \in \mathbb N$ let $\theta$ be a k-th root of unity. Consider the operator $U_\theta f(z) = f(\theta z)$. The following result can then be proven analogously to the previous theorems:
\begin{thm}
    Let $\theta$ be a $k$-th root of unity, $A \in \mathcal C_1(\mathcal H)$ and $m \in \{ 0, 1, \dots, k-1\}$ such that $U_\theta A U_\theta^\ast = \theta^m A$. Then, whenever $A$ is Fredholm, it holds true that:
    \begin{align*}
        \ind(A) \in k\mathbb Z +m.
    \end{align*}
\end{thm}
We spare the reader with the details of the matrix algebra utilized for proving the result, as they are entirely analogous to the case $k = 2$ proven above in detail. Let us just mention that one first proves the case $m = 0$ by considering the matrix decomposition of the operators $T_{z/|z|}^j A T_{\overline{z}/|z|}^j$, $j = 1, \dots, k-1$, which shows that all the diagonal entries of $A$ have the same index whenever $U_\theta A U_\theta^\ast = A$. The general case is then deduced from the special case $m = 0$ by computing the index of the operator $AT_{z/|z|}^{k-m}$.

\section{Operators with strongly continuous modulation action}

In this section, we will again work in the picture $\mathcal H = F^2(\mathbb C^n)$. We recall, e.g.~from the previous section, that the algebra $\mathcal C_1(\mathcal H)$ plays a crucial role in the operator theory arising from quantum harmonic analysis. Of course, one could try to consider variants of this algebra. For example, one could set (our choice of notation will be clearer a little later in this text):
\begin{align*}
    \mathcal C_{-1}(\mathcal H) = \{ A \in \mathcal L(F^2): ~\| W_z A W_z - A\| \to 0, ~z \to 0\}.
\end{align*}
Considering this space is not entirely out of thin air, as the action of the dual phase space is implemented by the following ``modulation'' of operators (cf.~\cite{Berge_Berge_Fulsche}):
\begin{align*}
    \gamma_z(A) = W_{z/2} A W_{z/2}, \quad A \in \mathcal L(\mathcal H), ~z \in \mathbb C^n.
\end{align*}
Hence, $\mathcal C_{-1}(\mathcal H)$ consists precisely of those operators on which the modulation acts strongly continuous. 

A first crucial difference between $\mathcal C_1(\mathcal H)$ and $\mathcal C_{-1}(\mathcal H)$ is the fact that $\mathcal C_{-1}(\mathcal H)$ is no longer an algebra. Instead, we have the following fact, which is straigthforward to prove:
\begin{lem}
    $\mathcal C_{-1}(\mathcal H)$ is a closed two-sided module over $\mathcal C_{1}(\mathcal H)$. 
\end{lem}
Besides having this apparently interesting structure, it is somewhat of a disappointment that the study of properties of $\mathcal C_{-1}(\mathcal H)$ immediately reduces to the study of $\mathcal C_1(\mathcal H)$, by the following very simple fact:
\begin{lem}
    $\mathcal C_{-1}(\mathcal H) = U\mathcal C_1(\mathcal H) = \mathcal C_1(\mathcal H)U$.
\end{lem}
At first glance it seems that this would already be the end of any investigations regarding the properties of $\mathcal C_{-1}(\mathcal H)$. But nevertheless, there is at least one non-trivial question regarding $\mathcal C_{-1}(\mathcal H)$, namely: What is its intersection with $\mathcal C_1(\mathcal H)$? Here's the answer:

\begin{thm}\label{thm:intersection}
    $\mathcal C_1(\mathcal H) \cap \mathcal C_{-1}(\mathcal H) = \mathcal K(\mathcal H)$.
\end{thm}
\begin{proof}
Since $\mathcal K(\mathcal H) \subset \mathcal C_1(\mathcal H)$ and $U\mathcal K(\mathcal H) = \mathcal K(\mathcal H)$, it clearly follows that $\mathcal K(\mathcal H) \subset \mathcal C_1(\mathcal H) \cap \mathcal C_{-1}(\mathcal H)$. 

We therefore need to prove that every operator in the intersection is compact. For doing so, we recall that an operator $A \in \mathcal L(\mathcal H)$ is called \emph{sufficiently localized} provided there exists some $C > 0$ and $\beta > 2n$ such that:
\begin{align*}
    |\langle A k_z, k_w\rangle| \leq \frac{C}{(1 + |z-w|)^{\beta}}.
\end{align*}
By a theorem of Xia \cite{Xia}, combined with \cite[Theorem 3.1]{Fulsche2020}, sufficiently localized operators are dense in $\mathcal C_1(\mathcal H)$. Hence, operators of the form $UA$, with $A$ sufficiently localized, are dense in $\mathcal C_{-1}(\mathcal H)$. Now, for $UA \in \mathcal C_{-1}(\mathcal H)$ with $A$ sufficiently localized, we have for its Berezin transform:
\begin{align*}
    |\widetilde{UA}(z)| = |\langle UA k_z, k_z\rangle| &= |\langle Ak_z, k_{-z}\rangle| \leq \frac{C}{(1 + 2|z|)^\beta},
\end{align*}
in particular $\widetilde{UA} \in C_0(\mathbb C^n)$ (the continuous functions vanishing at infinity). Since the Berezin transform $B \mapsto \widetilde{B}$ maps $\mathcal L(\mathcal H)$ continuously to $C_b(\mathbb C^n)$ (the $C^\ast$-algebra of bounded and continuous functions), we conclude that the Berezin transform of every operator in $\mathcal C_{-1}(\mathcal H)$ is contained in $C_0(\mathbb C^n)$. An application of the compactness characterization of operators on the Fock space (\cite{Bauer_Isralowitz2012}, see also \cite{Fulsche2020}) yields that
\begin{align*}
    \mathcal C_1(\mathcal H) \cap \mathcal C_{-1}(\mathcal H) \subset \mathcal K(\mathcal H),
\end{align*}
which concludes the proof.
\end{proof}
There is a more general version of the above result, which we now want to deduce. 

For any unitary matrix $\Theta \in \mathcal U(\mathbb C^n)$, we consider the operator
\begin{align*}
    U_\Theta f(z) = f(\Theta z), ~z \in \mathbb C^n
\end{align*}
on $F^2(\mathbb C^n)$. Clearly, when $\Theta = \theta Id$, $\theta$ a root of unity, then the operator $U_\Theta$ agrees with the operator $U_\theta$ from the previous section.
\begin{rem}
    We want to note that the operators $U_{\theta I}$, $\theta \in \mathbb C$ with $|\theta| = 1$, appear also naturally in other settings: It is well-known that, by means of the Bargmann transform, they are unitarily equivalent to the \emph{fractional Fourier transforms} on $L^2(\mathbb R^n)$. In particular for $\theta = i$ the operator $U_{\theta I}$ is unitarily equivalent to the Fourier transform on $L^2(\mathbb R^n)$. See, for example, \cite{Zhu2019} for some relations between operators on $F^2(\mathbb C^n)$ and $L^2(\mathbb R^n)$ by means of the Bargmann transform. 
\end{rem}

We can now consider the space
\begin{align*}
    \mathcal C_\Theta(\mathcal H) := \{ A \in \mathcal L(\mathcal H): ~\| W_z A W_{-\Theta z} - A\|\to 0, ~z \to 0\}.
\end{align*}
With this notation, $\mathcal C_{-1}(\mathcal H) = \mathcal C_{-I}(\mathcal H)$ and $\mathcal C_1(\mathcal H) = \mathcal C_I(\mathcal H)$. Then, using the imminent relation $U_\Theta W_z = W_{\Theta z} U_\Theta$, one immediately obtains:
\begin{lem}
    Let $\Theta \in \mathcal U(\mathbb C^n)$. Then, the following holds true:
    \begin{align*}
        \mathcal C_{\Theta}(\mathcal H) = U_{\Theta^\ast} \mathcal C_1(\mathcal H) = \mathcal C_1(\mathcal H)U_{\Theta^\ast}.
    \end{align*}
\end{lem}

We aim now at extending Theorem \ref{thm:intersection} to a more general situation, namely, we want to give a characterization of $\mathcal C_{\Theta_1}(\mathcal H) \cap \dots \cap \mathcal C_{\Theta_d}(\mathcal H)$, where $\Theta_1, \dots, \Theta_d \in \mathcal U(\mathbb C^n)$. 

As a first step, we need two preparatory lemmas.
\begin{lem}
    Let $\Phi_1, \dots, \Phi_d \in \mathcal U(\mathbb C^n)$. Then,
    \begin{align*}
        \mathcal C_{\Theta_1}(\mathcal H) \cap \dots \cap \mathcal C_{\Theta_d}(\mathcal H) = \mathcal C_{\Theta_1 \Theta_d^{\ast}}(\mathcal H) \cap \dots \cap \mathcal C_{\Theta_{d-1}\Theta_d^{\ast}}(\mathcal H) \cap \mathcal C_{1}(\mathcal H)U_{\Theta_d^\ast}
    \end{align*}
\end{lem}
The proof of the lemma is an immediate consequence of the lemma before. Note that this reduces the problem of characterization to the case where $\Theta_d = I$.

The second lemma is the following:
\begin{lem}
    Let $\Theta_1, \dots, \Theta_d \in \mathcal U(\mathbb C^n)$. Then, the algebra
    \begin{align*}
        \mathcal C_{\Theta_1}(\mathcal H) \cap \dots \cap \mathcal C_{\Theta_d}(\mathcal H) \cap \mathcal C_1(\mathcal H)
    \end{align*}
    is a closed and $\alpha$-invariant subspace of $\mathcal C_1(\mathcal H)$.
\end{lem}
\begin{proof}
    It is clear that the intersection is a closed subspace of $\mathcal C_1(\mathcal H)$. The only point, which might not be immediately clear, is the fact that it is $\alpha$-invariant. To show this, we only need to verify the statement for $d = 1$, i.e., that for any unitary $\Theta \in \mathcal U(\mathbb C^n)$ the intersection $\mathcal C_\Theta(\mathcal H) \cap \mathcal C_1(\mathcal H)$ is $\alpha$-invariant (because the intersection of $\alpha$-invariant subspaces is again $\alpha$-invariant).

    Hence, let $A \in \mathcal C_\Theta(\mathcal H) \cap \mathcal C_1(\mathcal H)$. It is clear that for $w \in \mathbb C^n$ we again have $\alpha_w(A) \in \mathcal C_1(\mathcal H)$, hence we need to show that $\alpha_w(A) \in \mathcal C_\Theta(\mathcal H)$. We compute:
    \begin{align*}
        \| W_z \alpha_w(A) W_{-\Theta z} - \alpha_w(A)\| &= \| W_z W_w A W_{-w} W_{-\Theta z} - W_w A W_{-w}\|\\
        &= \| e^{i\sigma(z,w) + i\sigma(-w, -\Theta z)} W_w W_z A W_{-\Theta z} W_{-w} - W_w A W_{-w}\|\\
        &= \| e^{i\sigma(z - \Theta z, w)} W_w W_z A W_{-\Theta z} W_{-w} - W_w A W_{-w}\|\\
        &\leq \| e^{i\sigma(z - \Theta z, w)}\alpha_w(W_z A W_{-\Theta z} - A)\| + \| (1-e^{i\sigma(z - \Theta z, w)}) \alpha_w(A)\|\\
        &= \| W_z A W_{-\Theta z} - A\| + |1-e^{i\sigma(z - \Theta z, w)}| \| A\|.
    \end{align*}
    For $z \to 0$, both of the above terms converge to zero, concluding the proof.
\end{proof}
By the above lemma, the space we want to characterize can be determined by the means of the correspondence theorem (see \cite{Werner1984}), which establishes a one-to-one correspondence between closed, translation-invariant subspaces of $\operatorname{BUC}(\mathbb C^n)$, the bounded and uniformly continuous functions on $\mathbb C^n$, and closed, $\alpha$-invariant subspaces of $\mathcal C_1(\mathcal H)$ (see also \cite{Fulsche2020} for more details). We will recall some details on this later.

The next steps for our program are the following lemmas:
\begin{lem}\label{lemma:function_decay}
    Let $X$ be a finite-dimensional inner product space and for $j = 1, \dots, d$ let $V_j \subset X$ be subspaces and $f_j: X \to [0, \infty)$ be functions satisfying
    \begin{align*}
        \limsup_{|v| \to \infty, v \in V_j} \sup_{w \in V_j^\perp} f_j(v+w) = 0.
    \end{align*}
    Set $f := \inf_{j=1, \dots, d} f_j$ and $V = \sum_{j=1}^d V_j$. Then, $f$ satisfies :
    \begin{align*}
        \limsup_{|v| \to \infty, v \in V} \sup_{w \in V^\perp} f(v+w) = 0.
    \end{align*}
\end{lem}
\begin{proof}
    Let $\varepsilon > 0$ and let $R > 0$ such that for each $j = 1, \dots, d$ and each $v \in V_j$ with $|v| > R$ and $w \in V_j^\perp$ it holds true that $f_j(v+w) < \varepsilon.$

    Now, let $v \in V$ with $|v| > \sqrt{d}R$. Denote by $P_j$ the orthogonal projection from $X$ onto $V_j$. Then, for at least one $j$ one necessary have $|P_j v| \geq |v|/ \sqrt{d}$. For notational simplicity, we assume that $|P_1 v| \geq |v|/\sqrt{d}$. In particular, $|P_1 v| \geq R$. Now, let $w \in V^\perp = \cap_{j=1}^d V_j^\perp$. In particular, $(I-P_1)v + w \in V_1^\perp$. Hence,
    \begin{align*}
        f(v+w) \leq f_1(v+w) = f_1(P_1v + ((I-P_1)v + w)) \leq \varepsilon.
    \end{align*}
    Since $\varepsilon > 0$ was arbitrary, this concludes the proof of the lemma.
\end{proof}

\begin{lem}\label{lemma:berezintransform_limit}
    Let $A \in \mathcal C_{\Theta_1}(\mathcal H) \cap \dots \cap \mathcal C_{\Theta_d}(\mathcal H) \cap \mathcal C_1(\mathcal H)$. Write:
    \begin{align*}
        V_0 = (\cap_{j=1}^d E_{\Theta_j}(1))^\perp,
    \end{align*}
    where $E_{\Theta_j}(1)$ denotes the eigenspace of $\Theta_j$ with respect to the eigenvalue $1$. Then,
    \begin{align*}
        \limsup_{|v| \to \infty, v \in V_0} \sup_{w \in V_0^\perp} |\widetilde{A}(v+w)| = 0.
    \end{align*}
\end{lem}
\begin{proof}
    Let $j = 1, \dots, d$, $\varepsilon > 0$ and let $B \in \mathcal C_1(\mathcal H)$ be sufficiently localized such that $\|A - BU_{\Theta_j^\ast}\| < \varepsilon$.
    Then, for $z = v + w$ with $v \in E_{\Theta_j}(1)^\perp$ and $w \in E_{\Theta_j}(1)$ we have:
    \begin{align*}
        |\widetilde{BU_{\Theta_j^\ast}}(z)| &= |\langle Bk_{\Theta_j^\ast z},k_z\rangle| \leq \frac{C}{(1+|\Theta_j^\ast z - z|)^\beta}\\
        &= \frac{C}{(1+|(\Theta_j^\ast - I) v|^\beta} \to 0, \quad |v| \to \infty.
    \end{align*}
    Note that this is uniform in $w \in E_{\Theta_j}(1)$. Hence, we see that (writing again $z = v+w$ with $w \in E_{\Theta_j}(1)$, $v \in E_{\Theta_j}(1)^\perp$):
    \begin{align*}
        \limsup_{|v|\to \infty} \sup_{w \in E_{\Theta_j}(1)} |\widetilde{A}(v+w)| \leq \limsup_{|v|\to \infty} \sup_{w \in E_{\Theta_j}(1)} |\widetilde{BU_{\Theta_j^\ast}}(v+w)| + \varepsilon = \varepsilon.
    \end{align*}
    Since $\varepsilon > 0$ was arbitrary, we obtain:
    \begin{align*}
        \limsup_{|v|\to \infty} \sup_{w \in E_{\Theta_j}(1)} |\widetilde{A}(v+w)|= 0.
    \end{align*}
    Letting now $f_j(z) = \sup_{w \in V_0^\perp} |\widetilde{A}(P_j(z))|$, where $P_j$ is the orthogonal projection onto $E_{\Theta_j}(1)^\perp$, we see that with $V_j = E_{\Theta_j}(1)^\perp$ the assumptions of Lemma \ref{lemma:function_decay}, i.e.,
    \begin{align*}
        \limsup_{|v|\to \infty, v \in V_j} \sup_{w \in V_j^\perp} f_j(v+w) = 0
    \end{align*}
    are satisfied. Hence, Lemma \ref{lemma:function_decay} (combined with the identity $\sum_{j=1}^d E_{\Theta_j}(1)^\perp = \left( \cap_{j=1}^d E_{\Theta_j}(1)\right)^\perp$) yield the claim that we wanted to prove.
\end{proof}
\begin{lem}\label{lem:correspondence1}
    Let $A \in \mathcal C_{\Theta_1}(\mathcal H) \cap \dots \cap \mathcal C_{\Theta_d}(\mathcal H) \cap \mathcal C_1(\mathcal H)$ and let $V_0$ as in the previous lemma. Then, $\widetilde{A} \in C_0(V_0) \otimes \operatorname{BUC}(V_0^\perp)$. 
\end{lem}
\begin{proof}
    We first note that, since commutative $C^\ast$-algebras are nuclear, we do not need to worry about the choice of tensor product. Further, it is well-known that for two compact Hausdorff spaces $X$ and $Y$ it holds true that $C(X) \otimes C(Y) = C(X \times Y)$. In particular, writing $\alpha V_0$ for the one-point compactification of $V_0$ and $\mathcal M(\operatorname{BUC}(V_0^\perp))$ for the maximal ideal space of $\operatorname{BUC}(V_0^\perp)$ (which is a compactification of $V_0^\perp$), we see that 
    $$ (C_0(V_0) \oplus \mathbb C1) \otimes \operatorname{BUC}(V_0^\perp) = C(\alpha V_0 \times \mathcal M(\operatorname{BUC}(V_0^\perp))).$$ 
    If we can show that $\widetilde{A}$ extends to a continuous function on $\alpha V_0 \times \mathcal M(\operatorname{BUC}(V_0^\perp))$, then we have shown that $\widetilde{A} \in C_0(V_0) \otimes \operatorname{BUC}(V_0^\perp)$. Write $P_0$ for the orthogonal projection onto $V_0$. We let $(z_\gamma) \subset \mathbb C^n$ be a net in $\mathbb C^n$ which converges in the compactification. Then, we can decompose $z_\gamma$ as $z_\gamma = P_0 z_\gamma + (I-P_0) z_\gamma \in V_0 \oplus V_0^\perp$. Recall that $\alpha V_0 \times \mathcal M(\operatorname{BUC}(V_0^\perp))$ is a compactification of $\mathbb C^n$. Then, there are two possibilities: Either, we have $|P_0 z_\gamma| \to \infty$ or $|P_0 z_\gamma|$ remains bounded. In the first case, Lemma \ref{lemma:berezintransform_limit} shows that $\widetilde{A}(z_\gamma) = \widetilde{A}(P_0 z_\gamma + (I-P_0)z_\gamma) \to 0$. Hence, when $(P_0z_\gamma, (I-P_0)z_\gamma) \to (\infty, x)$, where $\infty$ denotes the point at infinity from $\alpha V_0$ and $x \in \mathcal M(\operatorname{BUC}(V_0^\perp))$, then $\widetilde{A}(z_\gamma) \to 0$ independently of the precise choice of the net. 

    In the second case, we necessarily have $P_0 z_\gamma \to z_0$ for some $z_0 \in V_0$. Let $\varepsilon > 0$ and $\gamma_0$ be such that for $\gamma \geq \gamma_0$ we have 
    \begin{align*}
        |\widetilde{A}(P_0 z_\gamma + (I-P_0)z_\gamma) - \widetilde{A}(z_0 + (I - P_0)z_\gamma)| \leq \varepsilon.
    \end{align*}
    Denote the limit of $(I - P_0)z_\gamma$ in $\mathcal M(\operatorname{BUC}(V_0^\perp))$ by $y$. Then, since the function $V_0^\perp \ni w \mapsto \widetilde{A}(z_0 + w)$ is contained in $\operatorname{BUC}(V_0^\perp)$, there exists $\gamma_1$ such that for $\gamma \geq \gamma_1$ it is:
    \begin{align*}
        |\widetilde{A}(z_0, (I-P_0)z_\gamma) - \widetilde{A}(z_0, y)| \leq \varepsilon.
    \end{align*}
    Since there exists $\gamma_2$ such that $\gamma_2 \geq \gamma_0$ and $\gamma_2 \geq \gamma_1$, we see that for $\gamma \geq \gamma_1$ the following holds true:
    \begin{align*}
        |\widetilde{A}(z_\gamma) - \widetilde{A}(z_0, y)| \leq 2 \varepsilon.
    \end{align*}
    In particular, $\widetilde{A}(z_\gamma) \to \widetilde{A}(z_0, y)$. Since this convergence depends only on the limit points of $P_0z_\gamma$ and $(I-P_0)z_\gamma$, and not on the precise net, Bourbaki's extension theorem (\cite[Page 81, Theorem I]{Bourbaki1987}) shows that $\widetilde{A}$ extends to a continuous function on $\alpha V_0 \times \mathcal M(\operatorname{BUC}(V_0^\perp))$, which finishes the proof.
\end{proof}
\begin{lem}\label{lem:correspondence2}
    Let $\Theta_1,\dots, \Theta_d \in\mathcal U(\mathbb C^n)$ and $V_0$ as above. Let $f \in C_0(V_0) \otimes \operatorname{BUC}(V_0^\perp)$. Then, $T_f \in  \mathcal C_{\Theta_1}(\mathcal H) \cap \dots \cap \mathcal C_{\Theta_d}(\mathcal H) \cap \mathcal C_1(\mathcal H)$.
\end{lem}
\begin{proof}
    It clearly suffices to prove this for $f$ being an elementary tensor, $f = g \otimes h$ with $g \in C_0(V_0)$ and $h \in \operatorname{BUC}(V_0^\perp)$. Note that one can, analogously to $F^2(\mathbb C^n)$, define the Fock spaces over the complex vector spaces $V_0$ and $V_0^\perp$, which we denote by $F^2(V_0)$ and $F^2(V_0^\perp)$. On these spaces, one can in entirely analogous ways define the Toeplitz operators $T_g \in \mathcal L(F^2(V_0))$ and $T_h \in \mathcal L(F^2(V_0^\perp))$. Since $g \in C_0(V_0)$ it holds true that $T_g \in \mathcal K(F^2(V_0))$ and further, by \cite[Theorem 3.1]{Fulsche2020} (which identified $\mathcal C_1$ with the Toeplitz algebra) it is $T_h \in \mathcal C_1(F^2(V_0^\perp)$. Since Toeplitz operators over Fock spaces with tensor product symbols are the tensor products of the Toeplitz operators, we see that $T_f = T_g \otimes T_h \in \mathcal K(F^2(V_0)) \otimes \mathcal C_1(F^2(V_0^\perp))$. 

    Based on this observation, it suffices to prove that $K \otimes B \in \mathcal C_{\Theta_j}(\mathcal H) \cap \mathcal C_1(\mathcal H)$ whenever $K \in \mathcal K(F^2(V_0))$ is of rank one, i.e., $K = \varphi \otimes \psi$, and $B \in \mathcal C_1(F^2(V_0^\perp))$. Under these conditions, it is clear that $K \otimes B \in \mathcal C_1(\mathcal H)$ so that we only need to verify that $K \otimes B \in \mathcal C_{\Theta_j}(\mathcal H)$. 
    
    We decompose $z \in \mathbb C^n$ as $z = (z_1, z_2) \in V_0 \times V_0^\perp$. Observe that $\Theta_j|_{V_0^\perp} = I$. In particular, $\Theta_j$ leaves both $V_0$ and $V_0^\perp$ invariant. For simplicity of notation, we will also write $\Theta_j z_1$ instead of $\Theta_j|_{V_0} z_1$ in the following. We now compute:
     \begin{align*}
        \| W_z K &\otimes B W_{-\Theta_j z} - K \otimes B\|\\
        &= \| W_{(z_1,0)} W_{(0, z_2)} K \otimes B W_{-\Theta_j(0, z_2)} W_{-\Theta_j(z_1,0)} - K \otimes B\|\\
        &\leq \| [W_{z_1} K W_{-\Theta_j z_1}] \otimes [W_{z_2} B W_{-z_2}] - K \otimes B\|\\
        &\leq \| [W_{z_1} K W_{-\Theta_j z_1}] - K\| \| B\| + \| K\| \| [W_{z_2} B W_{-z_2}] - B\|.
    \end{align*}
    We now clearly have $\| [W_{z_2} B W_{-z_2}] - B\| \to 0$ as $z_2 \to 0$. Since we assumed that K is of rank one, $K = \varphi \otimes \psi$, we can also continue to estimate the other term:
    \begin{align*}
        \| [W_{z_1} K W_{-\Theta_j z_1}] - K\| &= \| [W_{z_1} \varphi] \otimes [W_{\Theta_j z_1} \psi] - \varphi \otimes \psi\|\\
        &\leq \| W_{z_1} \varphi - \varphi\| \| \psi\| + \| W_{\Theta_j z_1} \psi - \psi\| \| \varphi\| \to 0, \quad z_1 \to 0.
    \end{align*}
    This finishes the proof.    
\end{proof}
As the final step, we now recall some facts from the correspondence theory. It was proven in \cite{Werner1984} that there is a one-to-one correspondence between translation-invariant closed subspaces $\mathcal D_0$ of $\operatorname{BUC}(\mathbb C^n)$ and $\alpha$-invariant closed subspaces $\mathcal D_1$ of $\mathcal C_1(\mathcal H)$. The discussion in \cite{Fulsche2020} on correspondence theory related the correspondence of spaces with Toeplitz quantization and the Berezin transform: $\mathcal D_0$ and $\mathcal D_1$ are corresponding spaces if and only if $T_f \in \mathcal D_1$ for each $f \in \mathcal D_0$ and $\widetilde{A} \in \mathcal D_0$ for each $A \in \mathcal D_1$. 
Now, since both the Toeplitz operators and the Berezin transform respect tensor products, we have the following result:
\begin{lem}
    Let $V_0 \subset \mathbb C$ be a complex subspace. Further, let $\mathcal D_0 \subset \operatorname{BUC}(V_0)$ and $\mathcal D_1 \subset \mathcal C_1(F^2(V_0))$ be a pair of corresponding spaces and $\mathcal E_0 \subset \operatorname{BUC}(V_0^\perp)$ and $\mathcal E_1 \subset \mathcal C_1(F^2(V_0^\perp))$ be a pair of corresponding spaces. Then, $\mathcal D_0 \otimes \mathcal E_0 \subset \operatorname{BUC}(\mathbb C^n)$ and $\mathcal D_1 \otimes \mathcal E_1 \subset \mathcal C_1(F ^2(\mathbb C^n))$ are corresponding spaces.
\end{lem}
Since $C_0(V_0)$ corresponds to $\mathcal K(F^2(V_0))$ and $\operatorname{BUC}(V_0^\perp)$ is corresponding to $\mathcal C_1(F^2(V_0^\perp))$ (see \cite{Fulsche2020} for details), we in particular have proven that $C_0(V_0) \otimes \operatorname{BUC}(V_0^\perp)$ and $\mathcal K(F^2(V_0)) \otimes \mathcal C_1(F^2(V_0^\perp))$ are corresponding spaces. Further, lemmas \ref{lem:correspondence1} and \ref{lem:correspondence2} show that $C_0(V_0) \otimes \operatorname{BUC}(V_0^\perp)$ also corresponds to $\mathcal C_{\Theta_1}(\mathcal H) \cap \dots \cap \mathcal C_{\Theta_d}(\mathcal H) \cap \mathcal C_1(\mathcal H)$. Since the correspondence of spaces is unique, we have therefore proven the following theorem:
\begin{thm}
    Let $\Theta_1, \dots, \Theta_d \in \mathcal U(\mathbb C^n)$ and set $V_0 := (\cap_{j=1}^d E_{\Theta_j}(1))^\perp$. Then,
\begin{align*}
    \mathcal C_{\Theta_1}(\mathcal H) \cap \dots \cap \mathcal C_{\Theta_d}(\mathcal H) \cap \mathcal C_1(\mathcal H) = \mathcal K(F^2(V_0)) \otimes \mathcal C_1(F^2(V_0^\perp)).
\end{align*}
\end{thm}

\section{The operator-operator Fourier transform}
One of the basic tools of quantum harmonic analysis is the \emph{Fourier-Weyl transform}: For a trace class operator $A \in \mathcal T^1(\mathcal H)$ it is defined as $\mathcal F_W(A)(\xi) = \tr(AW_\xi^\ast)$. The inverse Fourier-Weyl transform is then formally defined as $\mathcal F_W^{-1}(f) = \int_\Xi f(\xi) W_\xi ~d\xi$. Here, $d\xi$ is an appropriately normalized Haar measure. 

We further recall the formula for the symplectic Fourier transform:
\begin{align*}
    \mathcal F_\sigma(f)(z) = \frac{1}{(2\pi)^n} \int_{\mathbb R^{2n}} f(w) e^{i\sigma(w,z)}~dw, \quad f \in \mathcal S(\mathbb R^{2n}).
\end{align*}
It is well-known that the symplectic Fourier transform $\mathcal F_\sigma: \mathcal S(\mathbb R^{2n}) \to \mathcal S(\mathbb R^{2n})$ extends to an operator on $L^2(\mathbb R^n)$, and equally well on $\mathcal S'(\mathbb R^n)$. Further, the Fourier-Weyl transform $\mathcal F_W$ yields a map from $\mathcal S(\mathcal H)$ (the Schwartz operators, cf. \cite{keyl_kiukas_werner16}, which can be identified with operators with integral kernel in $\mathcal S(\mathbb R^{2n})$) to $\mathcal S(\mathbb R^{2n})$, from $\mathcal T^2(\mathcal H)$ to $L^2(\mathbb R^{2n})$ and from $\mathcal S'(\mathcal H)$ (the tempered operators, dual space to $\mathcal S(\mathcal H)$) to $\mathcal S'(\mathbb R^{2n})$. On the level of these spaces, we can consider the composition $\mathcal F_W^{-1} \circ \mathcal F_\sigma$.

Recall that the Weyl quantization of a symbol $f \in \mathcal S(\mathbb R^{2n})$ is given by the operator:
\begin{align*}
    \operatorname{op}^W(f) \varphi(t) = \int_{\mathbb R^n} \int_{\mathbb R^n} f\left( \frac{x+t}{2}, \xi \right) e^{i\xi \cdot (t-x)} \varphi(x)~dx~d\xi.
\end{align*}
It is well-known (see, e.g., \cite[Lemma 5]{Fulsche_Luijk2025}) that the composition $\mathcal F_W^{-1} \circ \mathcal F_\sigma$ agrees with the Weyl quantization. Hence, on the level of $L^2$ we have the following diagram:
\[
    \begin{tikzcd}[arrows=rightarrow]
L^2(\mathbb R^n) \arrow[r, "\mathcal F_\sigma"] \arrow[dr, "\operatorname{op}^W"]&L^2(\mathbb R^n) \arrow[d, "\mathcal F_W^{-1}"]\\
&\mathcal T^2(\mathcal H)
\end{tikzcd}
\]
The analogous diagrams are true on the level of Schwartz functions and tempered distributions. 

It would be desirable to have some form of ``operator-to-operator Fourier transform'' $\mathcal F_{op}$ completing the diagram as follows:
\[
    \begin{tikzcd}[arrows=rightarrow]
L^2(\mathbb R^n) \arrow[r, "\mathcal F_\sigma"] \arrow[dr, "\operatorname{op}^W"] &L^2(\mathbb R^n) \arrow[d, "\mathcal F_W^{-1}"]\\
\mathcal T^2(\mathcal H) \arrow[u, "\mathcal F_W"] \arrow[r, "\mathcal F_{op}"] &\mathcal T^2(\mathcal H)\\
\end{tikzcd}
\]
Of course, it is now rather clear how to complete the diagram: Set $\mathcal F_{op} = \mathcal F_W^{-1} \circ \mathcal F_\sigma \circ \mathcal F_W$. With this choice, it is also clear that there are analogous diagrams on the level of Schwartz functions / operators and tempered distributions / operators.

The obvious task is now to actually compute what $\mathcal F_{op}$ looks like. Here's the answer:
\begin{thm}\label{thm:fourier}
    For $A \in\mathcal S'(\mathcal H)$ it is $\mathcal F_{op}(A) = \frac{1}{2^n} AU$.
\end{thm}
Before proving the theorem, we also need the following lemma. The statement is well-known, but for completeness we provide the proof:
\begin{lem}
    It is $\mathcal F_W(U) = \frac{1}{(4\pi)^n}$.
\end{lem}
\begin{proof}
    The proof consists in two steps:
    \begin{enumerate}
        \item $\operatorname{op}^W(\delta_0) = 2^n U$. To check this, we simply compute by definition for $f \in \mathcal S(\mathbb R^{2n})$ and $\varphi \in \mathcal S(\mathbb R^n)$:
        \begin{align*}
            \operatorname{op}^W(f) \varphi(t) &= \int_{\mathbb R^n} \int_{\mathbb R^n} f\left( \frac{x+t}{2}, \xi \right) e^{i\xi \cdot (t-x)} \varphi(x)~dx~d\xi\\
            &= \int_{\mathbb R^n} \int_{\mathbb R^n} f\left(\frac{x}{2}, \xi\right) e^{i\xi \cdot (2t - x)} \varphi(x-t)~dx~d\xi\\
            &= 2^n \int_{\mathbb R^n} \int_{\mathbb R^n} f(x,\xi) e^{i\xi \cdot (2t - 2x)} \varphi(2x - t) ~dx~d\xi.
        \end{align*}
        With this representation of the Weyl pseudodifferential operator, it is not hard to verify by a density argument that for $f = \delta_0$ we get (where the equalities have to be interpreted in the weak sense):
        \begin{align*}
            \operatorname{op}^W(\delta_0)\varphi(t) &= 2^n \int_{\mathbb R^{n}} \int_{\mathbb R^n} e^{i\xi \cdot (2t - 2x)} \varphi(2x - t) ~d\delta_0(x) ~d\delta_0(\xi)\\
            &= 2^n \int_{\mathbb R^{n}} e^{i\xi \cdot 2t} \varphi(-t)~d\delta_0(\xi) = 2^n\varphi(-t) = 2^n U\varphi(t).
        \end{align*}
        We want to emphasize that this relation been known for a much longer time, see for example \cite{Grossmann1976}.
        \item $\mathcal F_\sigma(\delta_0) = \frac{1}{(2\pi)^n}$. This follows right from the definition, which also has to be interpreted in the appropriate weak sense:
        \begin{align*}
            \mathcal F_\sigma(\delta_0) = \frac{1}{(2\pi)^n} \int_{\mathbb R^{2n}} e^{i\sigma(z,w)}~d\delta_0(z) = \frac{1}{(2\pi)^n}
        \end{align*}
        \item To arrive at the final result, note that 
        \begin{align*} 
        \mathcal F_W(U) = \frac{1}{2^n} \mathcal F_W \circ \mathcal F_W^{-1} \circ \mathcal F_\sigma(\delta_0) = \frac{1}{2^n} \mathcal F_\sigma(\delta_0) = \frac{1}{(4\pi)^n}.
        \end{align*}
        This finishes the proof.\qedhere
    \end{enumerate}
\end{proof}
\begin{proof}[Proof of Theorem \ref{thm:fourier}]
    We are going to prove the statement for $A$ being a Schwartz operator. By density, the statement then extends to all tempered operators, in particular to all bounded operators.

We recall from \cite{Werner1984} (see also \cite[Corollary 6.21]{Fulsche_Galke2025}) that for $A, B \in \mathcal T^1(\mathcal H)$ we have the following identity:
\begin{align}\label{id:twistedconv}
    \mathcal F_W(BA) = \mathcal F_W(B) \ast_\sigma \mathcal F_W(A).
\end{align}
Here, $f \ast_\sigma g$ is the \emph{twisted convolution}:
\begin{align*}
    f \ast_\sigma g(\xi) = \int_\Xi f(\xi-w) g(w) e^{i\sigma(\xi, w)}~dw.
\end{align*}
Note that, as discussed in \cite{keyl_kiukas_werner16}, Eq.~\eqref{id:twistedconv} holds true on a large general scale. For example, when $A$ is a Schwartz operator and $B$ is a tempered operator (in the sense of \cite{keyl_kiukas_werner16}), this identity is still valid. Hence, for $A \in \mathcal S(\mathcal H)$ we can let $B = U$ to obtain:
\begin{align*}
    \mathcal F_W(AU)(\xi) &= \mathcal F_W(A) \ast_\sigma \mathcal F_W(U)(\xi)\\
    &= \frac{1}{(4\pi)^n}\int_{\mathbb R^{2n}} \mathcal F_W(A)(\xi - w) e^{i\sigma(\xi,w)}~dw\\
    &= \frac{1}{(4\pi)^n} \int_{\mathbb R^{2n}} \mathcal F_W(A)(w)e^{i\sigma(w, \xi)}~dw\\
    &= \frac{1}{2^n}\mathcal F_\sigma \mathcal F_W(A)(\xi).
\end{align*}
This concludes the proof.
\end{proof}
Recall that $\alpha_z(A) = W_z A W_z^\ast$ and $\gamma_z(A) = W_{\frac{z}{2}} A W_{\frac{z}{2}}$. Then, from the covariance properties of $\mathcal F_W$ and $\mathcal F_\sigma$, one immediately obtains (as one can now also check directly):
\begin{lem}
    Let $A \in \mathcal L(\mathcal H)$. Then, for $z \in \mathbb R^{2n}$: $\alpha_z(\mathcal F_{op}(A)) = \mathcal F_{op}(\gamma_{2z}(A))$ and $\gamma_z(\mathcal F_{op}(A)) = \mathcal F_{op}(\alpha_{z/2}(A))$. 
\end{lem}
\begin{rem}
    Formally, $\mathcal F_{op}$ also satisfies properties analogous to the properties of the symplectic Fourier transform and the Fourier-Weyl transform. Nevertheless, having the inclusion properties of the Schatten classes available ($\mathcal T^p(\mathcal H) \subset \mathcal T^q(\mathcal H)$ for $1 \leq p \leq q$), these properties are rather trivial. For example, the Riemann-Lebesgue lemma for $\mathcal F_{op}$ would state that $\mathcal F_{op}$ maps $\mathcal T^1(\mathcal H)$ to $\mathcal K(\mathcal H)$, which is of course trivial.
\end{rem}

 Theorem \ref{thm:fourier} now easily yields the following result. Note that related statements have been obtained in the literature (\cite{Toft2002, Bauer_Fulsche_Toft2025}), but those statements were always formulated for certain Banach ideals and the proofs had a different flavor.
\begin{thm}
    Let $\mathcal I \subset \mathcal L(\mathcal H)$ be any algebraic right-ideal of $\mathcal L(\mathcal H)$. Then, for $f \in\mathcal S'(\mathbb R^{2n})$, the following holds true:
    \begin{align*}
        \operatorname{op}^W(f) \in \mathcal I \Leftrightarrow \operatorname{op}^W(\mathcal F_\sigma(f)) \in \mathcal I.
    \end{align*}
\end{thm}
\begin{proof}
    If $A = \operatorname{op}^W(f)$, then $A \in \mathcal I$ clearly implies $\frac{1}{2^n} AU = \operatorname{op}^W(\mathcal F_\sigma(f)) \in \mathcal I$. Applying this reasoning once again yields the other implication.
\end{proof}
We give one example which shows that the class of ideals in the result is strictly larger than the class of Schatten-von Neumann ideals:
\begin{ex}\label{ex:left_ideal}
    Let $X \subset \mathcal H$ be a closed subspace and
    \begin{align*}
        \mathcal I_X^\ast := \{ A \in \mathcal L(\mathcal H): ~A^\ast|_X = 0\}.
    \end{align*}
    Then, $\mathcal I_X^\ast$ is a right-ideal, hence the previous result states that for any $f \in \mathcal S'(\Xi)$ with $\operatorname{op}^w(f)$ bounded it is $\operatorname{op}^w(f) \in \mathcal I_X^\ast$ if and only if $\operatorname{op}^w(\mathcal F_\sigma(f)) \in \mathcal I_X^\ast$.
\end{ex}
It is not hard to verify (and well-known) that $U \operatorname{op}^W(f) U = \operatorname{op}^W(\beta_-(f))$, where $\beta_-(f)(z) = f(-z)$. In particular, we see that:
\begin{lem}
    Let $f \in \mathcal S'(\mathbb R^{2n})$. Then, $\frac{1}{2^n}U\operatorname{op}^W(f) = \operatorname{op}^W(\mathcal F_\sigma(\beta_-(f)))$. 
\end{lem}
We therefore obtain the following two versions of the above theorem:
\begin{thm}
\begin{enumerate}
    \item Let $\mathcal I\subset \mathcal L(\mathcal H)$ be any algebraic left-ideal of $\mathcal L(\mathcal H)$. Then, for $f \in \mathcal S'(\mathbb R^{2n})$, the following holds true:
    \begin{align*}
        \operatorname{op}^W(f) \in \mathcal I \Leftrightarrow \operatorname{op}^W(\mathcal F_\sigma(\beta_-(f))) \in \mathcal I.
    \end{align*}
    \item Let $\mathcal I\subset \mathcal L(\mathcal H)$ be any algebraic two-sided ideal of $\mathcal L(\mathcal H)$. Then, for $f \in \mathcal S'(\mathbb R^{2n})$, the following holds true:
    \begin{align*}
        \operatorname{op}^W(f) \in \mathcal I \Leftrightarrow \operatorname{op}^W(\mathcal F_\sigma(f)) \in \mathcal I \Leftrightarrow \operatorname{op}^W(\beta_-(f)) \in \mathcal I \Leftrightarrow \operatorname{op}^W(\mathcal F_\sigma\beta_-(f)) \in \mathcal I.
    \end{align*}
\end{enumerate}
\end{thm}
\begin{ex}
    Similarly to Example \ref{ex:left_ideal}, one can construct left-ideals which are not summability ideals. For doing so, let $X \subset \mathcal H$ be a closed subspace and then consider
    \begin{align*}
        \mathcal I_X := \{ A \in \mathcal L(\mathcal H): ~A|_X = 0\}.
    \end{align*}
    Then, $\mathcal I_X$ is a left ideal of $\mathcal L(\mathcal H)$. Similarly, if $Y \subset \mathcal H$ is also a closed subspace, then $\mathcal I_Y \cap \mathcal I_X^\ast$ is a two-sided ideal of $\mathcal L(\mathcal H)$ which is not defined by a summability condition of singular values.
\end{ex}

\begin{rem}
    We want to end this discussion by mentioning that the decisive factor in this discussion was that we are working within the Weyl quantization. Analogous results do not hold true, for example, for the Kohn-Nirenberg quantization. Nevertheless, entirely analogous reasoning can be employed in the setting of quantum harmonic analysis over the phase space $\Xi = G \times \widehat{G}$ when $G$ is a 2-regular locally compact abelian group, that is, a locally compact abelian group on which the map $g \mapsto g+g$ is an isomorphism of topological groups. In this setting, one can properly define the Weyl operators and the Weyl calculus, see \cite[Example 2.7]{Fulsche_Galke2025} for the precise setting leading to the connection of the Weyl calculus and quantum harmonic analysis. The key identity $\operatorname{op}^W(\delta_0) = c U$ is still satisfied for some constant $c > 0$, which was the key step in the above reasoning. 
\end{rem}

\bibliographystyle{amsplain}
\bibliography{References}

%
\noindent Robert Fulsche\\
\href{fulsche@math.uni-hannover.de}{\Letter fulsche@math.uni-hannover.de}\\
Hannover\\
GERMANY

\end{document}